\documentclass[a4paper,draft,11pt]{article}
\usepackage{array}
\usepackage{theorem}
\usepackage{amsmath,amscd,amssymb}
\usepackage{latexsym}
\usepackage{epsfig}
\usepackage{pb-diagram}
\theorembodyfont{\sl}

\newtheorem{lemma}{Lemma}[section]

\newtheorem{proposition}[lemma]{Proposition}
\newtheorem{theorem}[lemma]{Theorem}
\newtheorem{corollary}[lemma]{Corollary}

\newcommand{\BB}{\mathbb B}
\newcommand{\CC}{\mathbb C}

\newcommand{\HH}{\mathbb H}

\newcommand{\NN}{\mathbb N}

\newcommand{\QQ}{\mathbb Q}
\newcommand{\RR}{\mathbb R}

\newcommand{\TT}{\mathbb T}

\newcommand{\ZZ}{\mathbb Z}

\newcommand{\cL}{\mathcal L}

\newcommand{\To}{\longrightarrow}

\newcommand{\tensor}{\otimes}

\renewcommand{\Bar}{\overline}

\newcommand{\imic}{\cong}

\newcommand{\half}{{\frac{1}{2}}}

\newcommand{\PSL}{\mathop{\mathrm {PSL}}\nolimits}

\newcommand{\Sp}{\mathop{\mathrm {Sp}}\nolimits}
\newcommand{\SU}{\mathop{\mathrm {SU}}\nolimits}

\newcommand{\Sy}{\mathop{\null\mathrm {S}}\nolimits}
\newcommand{\Ut}{\mathop{\null\mathrm {U}}\nolimits}

\newcommand{\Aut}{\mathop{\mathrm {Aut}}\nolimits}

\newcommand{\Ker}{\mathop{\mathrm {Ker}}\nolimits}

\newcommand{\Lie}{\mathop{\mathrm {Lie}}\nolimits}

\newcommand{\Span}{\mathop{\mathrm {Span}}\nolimits}
\newcommand{\Ad}{\mathop{\mathrm {Ad}}\nolimits}
\newcommand{\codim}{\mathop{\mathrm {codim}}\nolimits}

\renewcommand{\Im}{\mathop{\mathrm {Im}}\nolimits}
\newcommand{\Id}{\mathop{\mathrm {Id}}\nolimits}

\newcommand{\rk}{\mathop{\mathrm {rk}}\nolimits}

\newcommand{\id}{\mathop{\mathrm {id}}\nolimits}

\newcommand{\ab}{\mathop{\mathrm {ab}}\nolimits}

\newcommand{\bde}{\mathbf e}

\newcommand{\qedsymbol}{\mbox{$\Box$}}
\newcommand{\qed}{\unskip\nobreak\hfil\penalty50\hskip1em\hbox{}\nobreak
\hfill\qedsymbol\parfillskip=0pt\finalhyphendemerits=0}
\newenvironment{proof}{\begin{ProofwCaption}{Proof}}{\end{ProofwCaption}}
\newenvironment{ProofwCaption}[1]
{\addvspace\theorempreskipamount \noindent{\it #1.}\rm}
{\qed \par \addvspace\theorempostskipamount}

\newcommand{\GmU}{\Upsilon}
\newcommand{\GmL}{\Lambda}
\newcommand{\GFix}{\Phi}
\newcommand{\DGmStor}{(D/\Gamma)'_\Sigma}

\newcommand{\ParGm}{\mathop{\rm MPar}\nolimits_\Gamma}
\newcommand{\SmQ}{{\Sigma(Q)}}
\newcommand{\SmP}{{\Sigma(P)}}
\newcommand{\fA}{{\bf R}}
\newcommand{\fG}{{\bf G}}
\newcommand{\fH}{{\bf H}}
\newcommand{\fL}{{\bf L}}
\newcommand{\comGm}{[\Gamma,\Gamma]}
\newcommand{\Fix}{{\rm Fix}}
\newcommand{\Nu}{{\bf N}}

\begin{document}

\title{Fundamental groups of toroidal compactifications}
\author{A.K. Kasparian\thanks{Research partially supported by Contract
    015/9.04.2014 with the the Scientific Foundation of Kliment
    Ohridski University of Sofia.}\\
Faculty of Mathematics and Informatics\\
Kliment Ohridski University of Sofia\\
5 James Bouchier blvd.,
1164 Sofia,
Bulgaria
\and G.K. Sankaran\\Department of Mathematical Sciences\\University of
Bath, Bath BA2 7AY, UK}
\date{      }

\maketitle

\begin{abstract}
We compute the fundamental group of a toroidal compactification of a
Hermitian locally symmetric space $D/\Gamma$, without assuming either
that $\Gamma$ is neat or that it is arithmetic. We also give bounds
for the first Betti number.
\end{abstract}

Many important complex algebraic varieties can be described as locally
symmetric varieties. Examples include modular curves $\HH/\Gamma$,
where $\HH$ is the upper half-plane and $\Gamma < \PSL(2, \ZZ)$;
classifying spaces for Hodge structures or (in cases where a Torelli
theorem holds) moduli spaces of polarised varieties, such as moduli of
abelian varieties and of K3 surfaces; and special surfaces, such as
Hilbert modular surfaces.

Locally symmetric varieties are in general non-compact, and we want to
be able to compactify them and to study the geometry of the
compactifications, especially the birational geometry, which does not
depend on the choice of compactification. We work with toroidal
compactifications as described in~\cite{AMRT}.

Two basic birational invariants of a compact complex manifold $X$ are
the Kodaira dimension $\kappa(X)$ and the fundamental group
$\pi_1(X)$.  There is an extensive literature on computing Kodaira
dimensions of specific locally symmetric varieties, which is usually
very difficult.

Computing the fundamental group is easier, but there are some gaps in
the literature which we aim to fill.  We study the fundamental group
of a toroidal compactification $\DGmStor$ of a non-compact, not
necessarily arithmetic quotient $D/\Gamma$ by a lattice $\Gamma$. In
general this is not a manifold, but it is normal and can be chosen to
have only quotient singularities. By~\cite[Sect~7]{Kollar} these do
not affect the fundamental group.

The main result of the article is Theorem~\ref{thm:fundamentalgroup},
describing $\pi_1(\DGmStor)$ as a quotient of the lattice $\Gamma$.

\section*{Acknowledgements} The first author is grateful to a referee for
pointing out a flaw in a previous version and for bringing to her
attention several references. We also thank Klaus Hulek and Xiuping Su
for useful remarks, and Maximilian R\"ossler for correcting an error.

\section{Background}\label{sect:background}

In this section we explain the background to the problem and establish
some terminology and notation.  A \emph{symmetric space} of
non-compact type is a quotient $D = G/K$ of a connected non-compact
semisimple Lie group $G$, assumed to be the real points of a linear
algebraic group defined over $\QQ$, by a connected maximal compact
subgroup $K$ of $G$.  If the centre of $K$ is not discrete, then $D$
carries a Hermitian structure and hence the structure of a complex
manifold, in fact a K\"ahler
manifold~\cite[Theorem~VIII.6.1.]{Helgason}.

By a \emph{lattice} in $G$ we mean a discrete subgroup of $G$ of
finite covolume with respect to Haar measure. A lattice $\Gamma$ is
said to be \emph{arithmetic} if $\Gamma\cap G(\ZZ)$ is of finite index
in both $\Gamma$ and $G(\ZZ)$. It is said to be \emph{neat} if the
subgroup of $\CC^*$ generated by all eigenvalues of elements of
$\Gamma$ is torsion free.

A \emph{locally symmetric variety} is the quotient of a Hermitian
symmetric space $D$ by a lattice $\Gamma < G$. If $D/\Gamma$ is
compact then $\Gamma$ is said to be \emph{cocompact} or
\emph{uniform}. Non-uniform lattices are very common, however, and it
is this case that we are concerned with. By~\cite{BailyBorel} (for
$\Gamma$ arithmetic) and~\cite{Mok} these quotients are always
algebraic varieties, not just complex analytic spaces.

Toroidal compactifications of arithmetic quotients of
Hermitian symmetric spaces are constructed and described
in~\cite{AMRT}. Margulis~\cite{Margulis} showed that an irreducible
lattice $\Gamma$ in a semisimple Lie group $G$ is necessarily
arithmetic if $G$ has real rank $\rk_\RR G > 1$, and Garland and
Raghunathan~\cite{GR} had earlier considered the case of $\rk_\RR G=1$
and constructed a fundamental domain for the action of any lattice
$\Gamma$ on $D=G/K$.

A parabolic subgroup $Q<G$ is said to be \emph{$\Gamma$-rational}, for
$\Gamma$ a lattice (arithmetic or not) in $G$, if $\Gamma\cap N_Q$ is
a lattice in the unipotent radical $N_Q$ of $Q$. An analytic boundary
component $F$ of $D$ is said to be a $\Gamma$-rational boundary
component if its stabiliser is a $\Gamma$-rational parabolic
subgroup. The construction in~\cite{GR} shows that $D/ \Gamma$ can
be compactified by adjoining the $\Gamma$-quotients of the
$\Gamma$-rational boundary components, and this was later used by
Mok~\cite{Mok} to extend the construction of~\cite{AMRT} to the case
of non-arithmetic quotients.

The only Hermitian symmetric spaces whose automorphism groups have
real rank~$1$ are the complex balls $\BB^n = \SU(n,1)/\Sy(\Ut(n)\times
\Ut(1))$, so Mok's generalisation is needed only for ball quotients,
but they are of great interest. The construction of~\cite{AMRT}
becomes much simpler when $\rk_\RR G=1$, as we explain below.

An extensive reference on toroidal and other compactifications of
locally symmetric spaces $D/\Gamma$ is the monograph \cite{BorelJi} of
Borel and Ji.

The first results on the fundamental group of a smooth
compactification of a locally symmetric space concerned Siegel modular
$3$-folds, where $G=\Sp(2, \RR)$.  These are moduli spaces of abelian
surfaces and some such cases were studied in \cite{HeidrichKnoller},
in \cite{Knoller} and in \cite{HulekSankaran}.  More generally, the
fundamental group of a toroidal compactification of an arbitrary
Hermitian locally symmetric variety is studied in~\cite{Sankaran}.

For simplicity, we assume throughout this paper that $D$ is
irreducible: that is, that $G$ does not decompose into a product of
factors.

Denote by $\ParGm$ the set of $\Gamma$-rational maximal parabolic
subgroups of $G$. It is shown in~\cite{Sankaran} that if $\Gamma<G$ is a neat
arithmetic non-uniform lattice, then a toroidal
compactification $\DGmStor$ satisfies $\pi_1 \DGmStor = \Gamma/\GmU$,
where $\GmU$ is the subgroup of $\Gamma$ generated by the centres of
the unipotent radicals of all $Q\in\ParGm$. Moreover, in~\cite{GHS} it
is shown that there is a surjective group homomorphism $\Gamma \to
\pi_1(\DGmStor)$, whose kernel contains all $\gamma \in \Gamma$ with a
fixed point on $D$.

If $Q\in \ParGm$ then we denote the unipotent radical of $Q$ by
$N_Q$, the centre of $N_Q$ by $U_Q$ and the solvable radical of $Q$ by
$R_Q$. Let $\GmL$ be the subgroup of $\Gamma$ generated by all $\gamma
\in \Gamma \cap Q$ (for some $Q\in\ParGm$) for which some power
$\gamma^k \in R_Q$, and if $\gamma^k\in N_Q\subset R_Q$ then
$\gamma^k\in U_Q$. From this definition, $\GmL$ is normal in~$\Gamma$.

Our main result, Theorem~\ref{thm:fundamentalgroup}, is that
$\pi_1(\DGmStor) = \Gamma/\GmL \GmU$. This is a much more precise
version of a result of the second author from~\cite{Sankaran}, stated
here as Theorem~\ref{thm:Sankaran}. It is also slightly more general
than the results of~\cite{Sankaran} because, in the light
of~\cite{Mok}, we may now also allow non-arithmetic lattices.

Here is a synopsis of the paper.  Section~\ref{sect:parabolics}
introduces some notation and terminology and describes the structure
of maximal $\Gamma$-rational parabolic subgroups, largely
following~\cite{BorelJi}.  Section~\ref{sect:toroidal} describes the
toroidal compactifications $\DGmStor$ and their coverings
$(D/\Gamma_o)'_\Sigma$ for normal subgroups (not necessarily lattices)
$\Gamma_o \vartriangleleft \Gamma$ containing $\GmU$.
Section~\ref{sect:pi1,H1} comprises the main results of the article.

We show in Proposition~\ref{prop:interiorcompanionfixedpoints} that
any element $\gamma \GmU \in \Gamma/\GmU$ with a fixed point on
$(D/\GmU)'_\Sigma$ has a representative $\gamma \in \Gamma \cap Q$
with $\gamma ^k \in R_Q$ for some $Q \in \ParGm$ and $k \in \NN$, and
that $\gamma^k\in U_Q$ if $\gamma^k\in N_Q$. This suffices to prove
Theorem~\ref{thm:fundamentalgroup}, and from that we deduce bounds on
the first Betti numbers in Subsection~\ref{subsect:H1}.

\section{Parabolic subgroups}\label{sect:parabolics}

We collect here some properties of Hermitian symmetric spaces $D =
G/K$ of non-compact type and maximal parabolic subgroups $Q$ of
$G$. For more details see \cite[Chapter~1]{BorelJi}.

\subsection{Langlands decomposition of a parabolic
  subgroup}\label{subsect:langlands}

Any parabolic subgroup $Q$ of $G$ has a \emph{Langlands
  decomposition}~\cite[Equation~(I.1.10)]{BorelJi}
\[
Q = N_Q A_Q M_Q
\]
where $N_Q$ is the \emph{unipotent radical} of $Q$. We write $L_Q =
A_Q\times M_Q$, the \emph{Levi subgroup} of $Q$, and $R_Q = N_Q \rtimes A_Q$,
the \emph{solvable radical} of $Q$. The subgroup $A_Q$ is called the
\emph{split component} of $Q$, and $M_Q$ is a semisimple complement of
$R_Q$. All these groups are uniquely defined once we choose a maximal
compact subgroup $K$ of $G$.

We can refine this further if we assume, as we henceforth do, that $Q$
is a maximal parabolic subgroup of $G$. We denote by $U_Q$ the centre
of the unipotent radical $N_Q$ of $Q$. Since $N_Q$ is a $2$-step
nilpotent group, {i.e.} $[[N_Q, N_Q], N_Q] = 0$, we have $U_Q=[N_Q,
  N_Q]$, the commutator subgroup. We may identify $U_Q$ with its Lie
algebra $\Lie(U_Q)\imic \RR^m$, for $m = \dim_\RR U_Q$. The quotient
$V_Q = N_Q/U_Q$ is also an abelian group, naturally isomorphic to
$\CC^n$~\cite[(III.7.10)]{BorelJi} and $N_Q = U_Q \rtimes V_Q$ is a
semi-direct product of $U_Q$ and $V_Q$.

The semi-simple complement $M_Q$ of the solvable radical $R_Q$ of $Q$
is a product $M_Q = G'_{Q,l} \times G_{Q,h}$ of semisimple groups
$G'_{Q,l}$, $G_{Q,h}$ of noncompact type~\cite[(III.7.8)]{BorelJi}.

This gives us the \emph{refined Langlands decomposition}
\begin{equation}\label{eq:LangDecomp}
  Q = (U_Q \rtimes V_Q)\rtimes (A_Q \times G'_{Q,l} \times G_{Q,h})
\end{equation}
of an arbitrary maximal parabolic subgroup $Q$ of $G$. Note also that $G=QK$.

Let us describe briefly the group laws on $N_Q$ and $Q$. For this, and
later, it will be convenient to use superfix notation for conjugation:
$g^h=hgh^{-1}$. However, the letters $j$ and $k$ will always denote
integers, so $g^k$ means the $k$-th power of $g$.

There is an $\RR$-linear embedding $\eta \colon \Lie(V_Q) \to
\Lie(N_Q)$, whose image is the orthogonal complement of $\Lie(U_Q)$
with respect to the Killing form. The group law in $N_Q$ is
\begin{equation}\label{eq:GroupLawNQ}
  \exp(u_1+\eta(v_1))\exp(u_2+\eta(v_2)) =
  \exp \Big((u_1+u_2)+\eta(v_1+v_2)+\half[\eta(v_1),\eta(v_2)]\Big).
\end{equation}
For the group law in $Q$ we write $q_j=(n_j,l_j)\in Q = N_Q\rtimes
L_Q$, further decomposed as $l_j = (a_j,g_j,h_j)\in A_Q \times
G'_{Q,l} \times G_{Q,h}$ and $n_j = \exp(u_j+\eta(v_j))$ with $u_j \in
\Lie(U_Q)$ and $v_j \in \Lie(V_Q)$. Then
\begin{equation}\label{eq:GroupLawQ}
q_1q_2 = (n_1n_2^{l_1}, l_1l_2) = \Big(\exp\big(u+\eta(v)\big), a_1a_2, g_1g_2, h_1h_2\Big),
\end{equation}
where $u = u_1+u_2^{a_1g_1}+\half[\eta(v_1),\eta(v_2^{l_1})]$ and $v = v_1+v_2^{l_1}$.

Since $Q$ is maximal and $D$ is irreducible, the group $A_Q\imic (\RR_{>0}, .)$ is a
$1$-dimensional real torus of $G$ (see, for
instance,~\cite[Section~I.1.10]{BorelJi}).

The symmetric space $D$ has an embedding in a space $\check{D}$, the
compact dual, on which $G$ acts. The topological boundary of $D$ then
decomposes into complex analytic boundary components corresponding to
parabolic subgroups $Q$: namely, $Q$ is the normaliser of the boundary
component~$F(Q)$. See \cite[Proposition~I.5.28]{BorelJi} or
\cite[Proposition III.3.9.]{AMRT} for details.

If $F(P)\subseteq \Bar{F(Q)}$ then $U(P)\supseteq U(Q)$
by~\cite[Theorem~III.4.8(i)]{AMRT}.

\subsection{Horospherical decomposition}\label{subsect:refhoro}

For any parabolic subgroup we have $G = QK$~\cite[(I.1.20)]{BorelJi},
so $Q$ acts transitively on $D$; moreover, $Q \cap K = M_Q \cap K$. As
a result, the refined Langlands decomposition \eqref{eq:LangDecomp} of
$Q$ induces the \emph{refined horospherical decomposition}
\begin{equation}\label{eq:HorosphereDecomp}
  D = U_Q \times V_Q \times A_Q \times D'_{Q,l} \times D_{Q,h}
\end{equation}
of $D$ with $D'_{Q,l} = G'_{Q,l}/G'_{Q,l} \cap K$ and $D_{Q,h} =
G_{Q,h}/G_{Q,h} \cap K$: see~\cite[Lemma III.7.9]{BorelJi} and the
discussion there. The equality in \eqref{eq:HorosphereDecomp} is a
real analytic diffeomorphism.  The factors $D_{Q,h}\imic F(Q)$ and
$D'_{Q,l}$ are respectively Hermitian and Riemannian symmetric spaces
of noncompact type.

The group law in $Q$ induces a $Q$-action on the corresponding
horospherical decomposition.  If
$(n,l)=(\exp(u+\eta(v)),a,g,h) \in Q = N_Q \rtimes L_Q$
and $y=(n_y, x_y)= (u_y,v_y,a_y,z_y',z_y) \in D$ then
\begin{eqnarray*}\label{eq:HorosphericalAction}
  (n,l) (y) &=& (n n_y^l, lx_y)\\
  &=& \Big(\exp\big(u+ u_y^{ag}+\half[\eta(v),\eta(v_y^l)] + \eta(v+v_y^l)\big), a a_y, g z_y', hz_y\Big).
\end{eqnarray*}

\subsection{Siegel domains}\label{subsect:Siegeldomain}

In \cite{PyatetskiiShapiro} Pyatetskii-Shapiro realises the Hermitian
symmetric spaces $D = G/K$ of noncompact type as Siegel domains of
third kind.  These are families of open cones, parametrised by
products of complex Euclidean spaces and Hermitian symmetric spaces of
noncompact type.

In the refined horospherical decomposition \eqref{eq:HorosphereDecomp},
the $A_Q$-orbit
\begin{equation}\label{eq:coneasorbit}
 C_Q = A_Q D'_{Q,l} = \{ (a,z') \mid  a \in A_Q,  z'\in D'_{Q,l} \}
\end{equation}
of the Riemannian symmetric space $D'_{Q,l}$ is an open, strongly
convex cone in $U_Q \imic \RR^m$~\cite[Lemma III.7.7]{BorelJi}.  Note
that the reductive group $G_{Q,l} = A_Q \rtimes G'_{Q,l}$ acts
transitively on $C_Q$ since $C_Q = G_{Q,l}/G_{Q,l} \cap K =
G_{Q,l}/G'_{Q,l} \cap K$.

We embed $C_Q$ in the complexification $U_Q \otimes_\RR \CC \imic
(\CC^m, +)$ of $U_Q$ as a subset $i C_Q \subset iU_Q$ with pure
imaginary components.  Combining with \eqref{eq:HorosphereDecomp}, one
obtains a real analytic diffeomorphism of $D$ onto the product
\begin{equation}    \label{eq:SiegelDomain}
  (U_Q + i C_Q) \times V_Q \times D_{Q,h}.
\end{equation}
which will be called the Siegel domain realisation of $D$ associated
with $Q$. See~\cite{AMRT} for the relation between
\eqref{eq:SiegelDomain} and the classical Siegel domain presentation of
$D$.

In these coordinates, with notation as above, the action of $Q$ is given by
\begin{eqnarray}    \label{eq:DiffeomorphicSiegelDomainAction}
&&(u+\eta(v), a, g, h) (u_y + i(a_y, z_y'), v_y, z_y)\\
&&\qquad  =  \Big(u + u_y^{ag}+\half[\eta(v),\eta(v_y^l)] + i(aa_y, gz_y'), v + v_y^l, hz_y\Big)\nonumber
\end{eqnarray}
where $(u_y + i(a_y,z_y'),v_y,z_y) \in (U_Q + i C_Q) \times V_Q
\times D_{Q,h}$.

\section{Toroidal compactifications}\label{sect:toroidal}

We recall briefly enough detail on toroidal compactification for our
immediate purposes: for full details we refer to~\cite{AMRT}.

\subsection{Admissible fans and collections}\label{subsect:admissible}

Suppose that $Q\in\ParGm$: then $\GmU_Q =\Gamma \cap U_Q\imic\ZZ^m$ is
a lattice in $U_Q\imic \RR^m$.

We say that a closed polyhedral cone $\sigma \subset U_Q$ is
\emph{$\GmU_Q$-rational} if $\sigma=\RR_{\ge 0}u_1+\cdots+\RR_{\ge
  0}u_s$ for some $u_i\in\GmU_Q$.

A \emph{fan} (see~\cite{Fulton}) $\SmQ$ is a collection of closed
polyhedral cones in $U_Q$ such that any face of a cone in $\SmQ$ is
also in $\SmQ$ and any two cones in $\SmQ$ intersect in a common
face. It is \emph{$\GmU_Q$-rational} if all cones in $\SmQ$ are
$\GmU_Q$-rational.

The fan $\SmQ$ in $U_Q$ is said to be \emph{$\Gamma$-admissible} if it
is $\GmU_Q$-rational, it decomposes $C_Q$ (that is,
$C_Q\subseteq\bigcup_{\sigma\in\SmQ}\sigma$) and
$\Gamma_{Q,l}=\Gamma\cap G_{Q,l}$ acts on $\SmQ$ with only finitely
many orbits.

The lattice $\Gamma$ acts on $\ParGm$ by conjugation.  We say that a
family $\Sigma = \{\SmQ\}_{Q \in \ParGm}$ of $\Gamma$-admissible fans
$\SmQ$ is a \emph{$\Gamma$-admissible family} if:
\begin{itemize}
\item[(i)] $\gamma \SmQ = \Sigma(Q^\gamma)$, where $Q^\gamma=\gamma
  Q\gamma^{-1}$, for all
$\gamma \in \Gamma$ and $Q \in \ParGm$; and
\item[(ii)] $\SmQ = \{\sigma\cap U_Q\mid \sigma \in \SmP\}$ whenever
  $F(P) \subseteq \Bar{F(Q)}$, for $P,\, Q \in \ParGm$.
\end{itemize}

\subsection{Partial compactification at a cusp}\label{subsect:partcompact}

For $Q \in \ParGm$, the quotient $\TT(Q) = (U_Q \otimes_\RR
\CC)/\GmU_Q \imic (\CC^*)^m$ is an algebraic torus over $\CC$.

A $\Gamma$-admissible fan $\SmQ$ determines a toric variety $X_\SmQ$
that includes $\TT(Q)$ as a dense Zariski-open subset. More precisely,
$X_\SmQ$ is the disjoint union of all the quotients
$\TT(Q)_\sigma$ of $\TT(Q)$ by the complex algebraic tori $\Span_\CC
(\sigma) / (\Span_\CC (\sigma ) \cap \GmU_Q)$, corresponding to the  cones
$\sigma \in \SmQ$. In particular, the dense torus in $X_\SmQ$ is $\TT(Q)=\TT(Q)_{\{0\}}$.

Bearing in mind that $(U_Q + i C_Q)/\GmU_Q$ is an open subset of
$\TT(Q)$, we take the closure $\Bar{(U_Q + i C_Q)/\GmU_Q}$ of $(U_Q +
i C_Q)/\GmU_Q$ in $X_\SmQ$ and define $Y_\SmQ$ as the interior of
$\Bar{(U_Q + i C_Q)/\GmU_Q}$ in $X_\SmQ$. Here things are simpler in
the case of real rank~$1$: then $m=1$ and
$\Sigma(Q)=\{\{0\},\RR_{\geq 0}\}$ (there is no choice of fan to be made),
so $X_\SmQ = \CC$ and $Y _\SmQ$ is a disc.

The Siegel domain presentation~\eqref{eq:SiegelDomain} of $D$ associated
with $Q$ provides a real analytic diffeomorphism
\[
  D/\GmU_Q = (U_Q + i C_Q)/\GmU_Q \times V_Q \times D_{Q,h},
\]
and the $\Gamma$-admissible fan $\SmQ$ defines a partial
compactification
\begin{equation}   \label{eq:PartialCompletionParSubgr}
  (D/\GmU_Q)_\SmQ = Z_\SmQ = Y_\SmQ \times V_Q \times D_{Q,h}.
\end{equation}
By subdividing $\SmQ$ we may, and henceforth do, assume that $X_\SmQ$
and $Y_\SmQ$ are smooth: see~\cite{Fulton}.

To describe the $Q$-action on $Z_\SmQ$, consider the $\GmU_Q$-covering
map
\[
  \epsilon_Q\colon D = (U_Q + i C_Q) \times V_Q \times D_{Q,h}  \To
  D/\GmU_Q = (U_Q + i C_Q)/\GmU_Q \times V_Q \times D_{Q,h},
\]
given in the notation of \eqref{eq:coneasorbit} and
\eqref{eq:DiffeomorphicSiegelDomainAction} by
\[
 \epsilon_Q (u+ i  (a, \zeta'), \, v,  \, \zeta) = \left(\bde_Q(u+i
 (a, \zeta')), \,  v,  \, \zeta \right),
\]
where $\bde_Q\colon U_Q\tensor \CC \to \TT(Q)$ is the canonical map
with kernel $\GmU_Q$.  If we identify $\GmU_Q$ with $\ZZ^m$ then we
can identify $\bde_Q$ with exponentiation,
i.e.\ $\bde_Q(z_1,\ldots,z_m)=(e^{2\pi i z_1},\ldots,e^{2\pi i z_m})$
for $(z_1,\ldots,z_m)\in \CC^m$.

Identifying $N_Q$ with $\Lie(N_Q)$ and using \eqref{eq:DiffeomorphicSiegelDomainAction}, one expresses the action of
$\gamma=(u+\eta(v),a,g,h) \in \Lie(N_Q) \rtimes (A_Q \times G'_{Q,l} \times G_{Q,h})$ on
$y=(\bde_Q(u_y+i(a_y,z_y')),v_y,z_y) \in (U_Q + iC_Q)/\Upsilon_Q \times V_Q \times D_{Q,h}$ by
\begin{equation}\label{eq:quotientaction}
\gamma y
 = \Big(\bde_Q\big(u+ u_y^{ag} + \half[\eta(v),\eta(v_y^l)]+i(aa_y,gz_y')\big),
  v+ v_y^l,hz_y\Big).
\end{equation}
This $Q$-action extends by continuity to $Z_\SmQ$.

\subsection{The gluing maps}\label{subsect:gluingmaps}

For $P,\,Q\in\ParGm$ with $F(P)\subseteq \Bar{F(Q)}$, we are going to
describe explicitly the holomorphic map $\mu^Q_P\colon Z_\SmQ \To
Z_\SmP$ of \cite[Lemma~III.5.4]{AMRT}.

According to \cite[Theorem~III.4.8]{AMRT}, $U_Q$ is an $\RR$-linear
subspace of $U_P$. Therefore $\GmU_Q<\GmU_P$ and the identity map
\[
 \id_D\colon D = (U_Q + i C_Q) \times V_Q \times D_{Q,h}  \To D = (U_P
 + i C_P) \times V_P \times D_{P,h}
\]
induces a holomorphic covering
\[
\mu^Q_P\colon D/\GmU_Q = (U_Q + i C_Q)/\GmU_Q \times V_Q \times D_{Q,h}   \To
D/\GmU_P
\]
given by $\mu^Q_P (\GmU_Q x) = \GmU_P x$.

The inclusions $U_Q\tensor\CC \subset U_P\tensor\CC$ and $\GmU_Q <
\GmU_P$ induce a homomorphism $\mu^Q_{P,1} \colon \TT (Q) \To \TT(P)$,
which extends to $\mu^Q_{P,1}\colon X_\SmQ \To X_\SmP$, mapping
$Y_\SmQ$ into $Y_\SmP\subset\Bar{(U_P + i C_P)/\GmU_P}$. In this way,
one obtains a holomorphic gluing map
\[
 \mu^Q_P\colon Z_\SmQ = Y_\SmQ  \times V_Q \times D_{Q,h}   \To
 Y_\SmP  \times V_P \times D_{P,h} = Z_\SmP,
\]
given by
\[
 \mu^Q_P \left(\lim\limits_{t \to \infty} (y_t, v, z)\right) =
 \lim\limits_{t \to \infty} \mu^Q_P (y_t, v,z) = \lim\limits_{t \to
   \infty} (y_t + \GmU_P/\GmU_Q, v, z)
\]
where $y_t \in (U_Q + i C_Q)/\GmU_Q$ for $t\in\RR$ tends to some point
$\lim\limits_{t \to \infty} y_t \in Y_\SmQ$.  From this definition,
$\mu^Q_Q$ is the identity on $Z_\SmQ = (D/\GmU_Q)_\SmQ$, whether
$\Gamma$ is arithmetic or not. Again the real rank~$1$ case (in
particular, the non-arithmetic case) is simpler: then $F(Q)$ is a
point and there are no nontrivial inclusions $F(P)\subseteq
\Bar{F(Q)}$.

\subsection{Toroidal compactifications and coverings}\label{subsect:coverings}

We recall the construction of a toroidal compactification $\DGmStor$
of a locally symmetric variety $D/\Gamma$, associated with a
$\Gamma$-admissible family $\Sigma = \{\SmQ\}_{Q \in \ParGm}$ of fans
$\SmQ$ in $U_Q$.  In the notation of
subsection~\ref{subsect:partcompact}, consider the disjoint union
$\coprod_{Q \in \ParGm} Z_\SmQ$.

We denote by $\GmU$ the subgroup of $\Gamma$ generated by $\GmU_Q$ for
all $Q \in \ParGm$.  Suppose that $\Gamma_o$ is a normal subgroup of
$\Gamma$ containing $\GmU$. Its action on $D$ induces an equivalence
relation $\sim_{\Gamma_o}$ on $\coprod_{Q \in \ParGm}Z_\SmQ$, as in
the proof of \cite[Theorem~2.1]{Sankaran}: for $\Gamma_o=\Gamma$ it is
described in \cite[III.5]{AMRT}. Let $z_1\in Z_{\Sigma(Q_1)}$ and $z_2
\in Z_{\Sigma(Q_2)}$: then $z_1\sim_{\Gamma_o}z_2$ if there exist
$\gamma \in \Gamma_o$, $Q \in \ParGm$ and $z \in Z_\SmQ$, such that
$\Bar{F(Q)} \supseteq F(Q_1)$, $\Bar{F(Q)} \supseteq F(Q_2^\gamma)$,
$\mu_{Q_1}^Q (z) = z_1$ and $\mu^Q_{Q_2^\gamma} (z) = \gamma z_2$. In
the non-arithmetic case, $z_1 \sim _{\Gamma _o} z_2$ simply
reduces to the usual $\Gamma _o$-equivalence: $z_2 = \gamma z_1$ for
some $\gamma \in \Gamma _o$.

Then we put
\[
  (D/\Gamma_o)'_\Sigma = \left(\coprod_{Q \in \ParGm} Z_\SmQ \right)/\sim_{\Gamma_o}.
\]
In \cite{Sankaran} this is used to construct the
$(\Gamma/\GmU)$-Galois covering $(D/\GmU)'_\Sigma$ of $\DGmStor$ and
show that $(D/\GmU)'_\Sigma$ is a simply connected complex analytic
space. Notice that $\Gamma_o$ is not required to be a lattice, and
that $(D/\GmU)'_\Sigma$ is not compact.

In the proof of \cite[Theorem~1.5]{Sankaran} it is shown that
$Z_\SmQ$, which are diffeomorphic to $Y_\SmQ \times V_Q \times
D_{Q,h}$ for all $Q \in \ParGm$, are simply connected.  Further, the
proof of \cite[Theorem~2.1]{Sankaran} establishes that the natural
coverings $D/\GmU_Q \to D/\GmU$ extend to open holomorphic maps
$\pi^U_\SmQ\colon Z_\SmQ \to (D/\GmU)'_\Sigma$, which are
biholomorphic onto their images.

\section{The fundamental group and first Betti number}\label{sect:pi1,H1}

\subsection{The fundamental group}\label{subsect:pi1}

We begin by stating two theorems that summarise the results
from~\cite{Sankaran} and~\cite{GHS} on the fundamental group of a
toroidal compactification $\DGmStor$ of a quotient of $D=G/K$ by an
arithmetic lattice $\Gamma < G$.

\begin{theorem}\cite[Corollary~1.6, Theorem~2.1]{Sankaran}\label{thm:Sankaran}
 Let $D=G/K$ be a Hermitian
symmetric space and let $\Gamma$ be a non-uniform arithmetic lattice
in $G$.  Then the fundamental group $\pi_1(\DGmStor)$ of a toroidal
compactification $\DGmStor$ of $D/\Gamma$ is a quotient group of
$\Gamma/\GmU$.  In particular, if $\Gamma$ is a neat arithmetic
non-uniform lattice then $\pi_1(\DGmStor) = \Gamma/\GmU$.
\end{theorem}

The above is a more carefully stated version of the results
in~\cite{Sankaran}. Recall that $\GmU$ is the group generated by
$\GmU_Q$ for all $Q\in \ParGm$. In a similar style, we define $\GmL$
to be the subgroup of $\Gamma$ generated by all $\gamma \in \Gamma
\cap Q$ such that $\gamma^k\in R_Q$ for some $k\in \NN$ and some $Q
\in \ParGm$, and $\gamma^k\in U_Q$ if $\gamma^k\in N_Q$. (Another way
to say this is that $(\gamma N_Q)^k=\gamma^k N_Q$ is in
$R_Q/N_Q=A_Q \subset L_Q =Q/N_Q$, and if it is the identity then
$\gamma^k\in \Upsilon$.)

\begin{theorem}\cite[Lemma 5.2, Proposition 5.3]{GHS}\label{thm:GritsenkoHulekSankaran}
 Under the conditions of
Theorem~\ref{thm:Sankaran}, there is a commutative diagram
\[
\begin{diagram}
\node{\Gamma} \arrow{e,t}{\psi} \arrow{se,r}{\varphi} \node{\pi_1
  (D/\Gamma)}  \arrow{s}  \\
\node{\mbox{  }}   \node{\pi_1 \DGmStor}
\end{diagram}
\]
of surjective group homomorphisms, such that $\Ker\varphi$ and
$\Ker\psi$ contain all $\gamma \in \Gamma$ with a fixed point on $D$.
\end{theorem}

The following is the main result of the present paper. We use the
notation from Subsection~\ref{subsect:langlands}.

\begin{theorem}\label{thm:fundamentalgroup}
Let $D=G/K$ be an irreducible Hermitian symmetric space and let
$\Gamma$ be a non-uniform lattice in $G$. Then for any
$\Gamma$-admissible family $\Sigma$, the toroidal compactification
$\DGmStor$ has fundamental group
\[
\pi_1(\DGmStor) = \Gamma/ \GmL  \GmU.
\]
\end{theorem}

\begin{proof}
According to \cite{Sankaran}, $(D/\GmU)'_\Sigma$ is a path connected
simply connected locally compact topological space and $\Gamma/\GmU$
acts properly discontinuously on $(D/\GmU)'_\Sigma$ by
homeomorphisms. More precisely, $\gamma \GmU\colon (\GmU q) \mapsto
\GmU \gamma q$ defines a $\Gamma/\GmU$-action on $D/\GmU$, which
extends continuously to $(D/\GmU)'_\Sigma$. The quotient space
$(D/\GmU)'_\Sigma/(\Gamma/\GmU) = (D/\Gamma)'_\Sigma$ is the toroidal
compactification of $D/\Gamma$ associated with $\Sigma$.  By a theorem
of Armstrong~\cite{Armstrong}
\[
\pi_1 (\DGmStor) = (\Gamma/\GmU)/ (\Gamma/\GmU)^\Fix_o
\]
where $(\Gamma/\GmU)^\Fix_o$ is the subgroup of $\Gamma/\GmU$
generated by elements $\gamma \GmU$ with a fixed point on
$(D/\GmU)'_\Sigma$.

Theorem~\ref{thm:fundamentalgroup} therefore follows from
Proposition~\ref{prop:interiorcompanionfixedpoints}, which establishes
that $(\Gamma/\GmU)^\Fix_o = \GmL \GmU/\GmU$.
\end{proof}\medskip

In order to describe the action of $\Gamma/\GmU$ on $(D/\GmU)'_\Sigma$
note that the $\Gamma$-action on $\ParGm$ by conjugation determines
holomorphic maps $\gamma\colon Z_\SmQ \to Z_{\Sigma(Q^\gamma)} $ for
all $\gamma \in \Gamma$ and $Q \in \ParGm$.  Any $\gamma \in \Gamma$
transforms the $\sim_{\GmU}$-equivalence class of $z \in Z_\SmQ $ into
the $\sim_{\GmU}$-equivalence class of $\gamma z$, giving a
biholomorphic map $\gamma\colon (D/\GmU)'_\Sigma \To
(D/\GmU)'_\Sigma$.  By definition of $\sim_\GmU$, all $\gamma \in
\GmU$ act trivially on $(D/\GmU)'_\Sigma$ and the $\Gamma$-action on
$(D/\GmU)'_\Sigma$ reduces to a $(\Gamma/\GmU)$-action
\[
 (\Gamma/\GmU)  \times (D/\GmU)'_\Sigma \To (D/\GmU)'_\Sigma,
\]
given by
\begin{equation}\label{eq:GammaModUnipotentAction}
 (\gamma \GmU) \pi^U_\SmQ (z) = \pi^U_{\Sigma(Q^\gamma)}
(\gamma z)
\end{equation}
for $\gamma \GmU \in \Gamma/\GmU$ and $z \in Z_\SmQ$.

\begin{proposition}\label{prop:interiorcompanionfixedpoints}
In the notations from Theorem \ref{thm:fundamentalgroup}, a coset
$\gamma_0 \GmU \in \Gamma/\GmU$ has a fixed point on
$(D/\GmU)'_\Sigma$ if and only if for some $k \in \NN$ and some $Q \in
\ParGm$, there is a representative $\gamma \in \Gamma \cap Q $ of
$\gamma_0 \GmU = \gamma \GmU$ with $\gamma^k \in \Gamma \cap  R_Q$
and $\gamma ^k \in U_Q$ if $\gamma ^k \in N_Q$.
Hence the subgroup $(\Gamma/\GmU)^\Fix_o$ of $\Gamma/\GmU$ satisfies
$(\Gamma/\GmU)^\Fix_o = \GmL \GmU/\GmU$.
\end{proposition}

\begin{proof}
We first prove the ``only if'' part of the statement.

We claim that if $\gamma_0 \GmU \in \Gamma/\GmU$ has a fixed point on
\[
  (D/\GmU)'_\Sigma = \left(\coprod_{P \in \ParGm} Z_\SmP\right)/\sim_{\GmU}
\]
then there exist $\gamma \in \gamma_0 \GmU$ and $y \in Z_\SmQ$ for
some $Q \in \ParGm$, such that $\gamma y = y$. That is, if a coset
of $\GmU$ has a fixed point mod~$\GmU$ then some representative of
that coset has a fixed point ``on the nose''.

To prove this, notice that if $z_0 \sim_\GmU \gamma_0 z_0$ for some
$z_0 \in Z_\SmP$, then there exist $Q_1 \in \ParGm$, $z_1 \in
Z_{\Sigma(Q_1)}$ and $u_1 \in \GmU$ such that $F(P) \subseteq
\Bar{F(Q_1)}$ and $\mu^{Q_1}_P (z_1) = z_0$, and $F(P^{u_1\gamma_0})
\subseteq \Bar{F(Q_1)}$ and $\mu^{Q_1}_{P^{u_1\gamma_0}}(z_1) = u_1
\gamma_0 z_0$.

If $F(Q_1)=F(P)$, then $Q_1 =P$ and in~\eqref{eq:GammaModUnipotentAction}
we have $\mu^{Q_1}_{P^{u_1 \gamma_0}} = \mu^{Q_1}_P = \id_{Z_\SmP}$
and $z_0 = u_1 \gamma_0 z_0$. Since $\GmU$ is a normal subgroup of
$\Gamma$, we may take $\gamma = u_1 \gamma_0 \in \GmU \gamma_0 =
\gamma_0 \GmU$. In particular, this shows that the claim is true if
$F(P)$ is of maximal dimension.

Now we conclude the proof of the claim by induction on $\codim F(P)$:
suppose that it holds for all $P'\in\ParGm$ with $\dim
F(P')>\dim F(P)$, and take $Q_1$ as above. If $F(Q_1)=F(P)$ we are
done. If not, then $z_0 \sim_\GmU z_1$ and $\gamma_0 z_0 \sim_\GmU
\gamma_0 z_1$, because $\mu^{Q_1^{\gamma_0}}_{P^{\gamma_0}} (\gamma_0
z_1) = \gamma_0 z_0$. On the other hand, $\gamma_0 z_0 \sim_\GmU z_1$,
so $z_1 \sim_\GmU \gamma_0 z_1$ because $\sim_\GmU$ is an equivalence
relation. Thus $\gamma_0\GmU$ has the fixed point $z_1\in
Z_{\Sigma(Q_1)}$ so the claim follows by taking $P'=Q_1$.

Suppose then that $\gamma \in \Gamma$ has a fixed point $y \in
Z_\SmQ$ for some $Q \in \ParGm$.  Then $y = \gamma y \in
Z_{\Sigma(Q^{\gamma})}$ implies that $Q^{\gamma} = Q$; but the
parabolic subgroup $Q$ of $G$ coincides with its normaliser in $G$, so
$\gamma \in Q$.  We may therefore use the Langlands decomposition of
$Q$ and write
\[
\gamma = (\exp(u+\eta(v)), a, g, h) \in N_Q \rtimes (A_Q \times G'_{Q,l} \times G_{Q,h}).
\]
As above we take
\[
l = (a, g, h) \in L_Q = A_Q \times G'_{Q,l} \times G_{Q,h}.
\]
Any element of $X_\SmQ$ may be written as a limit of elements of
$\TT(Q)$, that is, as $\lim\limits_{t\to\infty}(\bde_Q(u_t+ix_t))$, and if the
element is in $Y_\SmQ$ then we may take $x_t\in C_Q$. So
\[
y = (\lim_{t\to\infty}\bde_Q(u_t+ix_t), v_y, z_y) \in Y_\SmQ
\times V_Q \times D_{Q,h}.
\]
Then by~\eqref{eq:quotientaction} and the continuity of the $Q$-action on
$(D/\GmU_Q)_\SmQ$, the condition $\gamma y=y$ is equivalent to
\begin{equation*}
\lim\limits_{t \to \infty} \bde_Q\Big(u+u_t^{ag}+\half[\eta(v),\eta(v_y^l)]+
 i(a,g)x_t \Big)  = \lim_{t\to\infty}\bde_Q(u_t+ix_t),
\end{equation*}
together with
\begin{equation*}
 v+v_y^l = v_y \text{ and }  h z_y  = z_y.
\end{equation*}
If $\gamma$ has a fixed point $y$ and $y \in D/\GmU_Q$ then $\gamma$
belongs to the compact stabiliser of $y$ in the isometry group $G$ of
$D$ and hence $\gamma$ is torsion. If $y \in Z_\SmQ \setminus (D
/\GmU_Q)$ we need to look at $g$ and $h$. For $h$ what we need is
immediate: it is in the stabiliser of $z_y\in D_{Q,h}$ and isotropy
groups in symmetric spaces are always torsion, so $h$ is of finite
order: by replacing $\gamma$ with a power we may assume that $h$ is
the identity.

For the Riemannian part $g$ a little more work is needed. The Levi
component $l=(a,g,\Id)$ of the element $\gamma$ has a fixed point $y'=
\lim\limits_{t \to \infty} {\bf e}_Q(u_t+ix_t)\in X_\SmQ \setminus
\TT(Q)$, so $y'\in \TT(Q)_\sigma$ for some unique $\sigma \in
\SmQ$. Therefore $l$ preserves $\sigma$. If we assume, as we may do,
that $Q$ has been chosen so as to maximise $\dim F(Q)$
(see~\cite[Lemma~III.5.5]{AMRT}), then $\sigma\cap C_Q\neq \emptyset$
(remember that $C_Q$ is an open cone but $\sigma$ is closed): this
follows from \cite[Theorem~III.4.8(ii)]{AMRT}.

Since $l$ preserves $\sigma$, it permutes the top-dimensional
cones of which $\sigma$ is a face: there are finitely many of these as
long as $\sigma\cap C_Q \neq \emptyset$. Therefore some power of
$l$ preserves a top-dimensional cone, so we may as well assume
that $\sigma$ is top-dimensional. The action of $l$ is thus
determined by its action on $\sigma = \sum_{i=1}^q \RR_{\ge 0}u_i$.

Thus $\gamma$ permutes the rays $\RR_{\ge 0}u_i$ (it may not fix them
pointwise) and therefore some power, in fact $l^{q!}$, fixes all
the rays, so we may as well assume that $l$ fixes all the
rays. In particular it fixes a rational basis of $U_Q$ up to
scalars. Now consider the real subgroup of $G_{Q,l}$ that fixes that basis
up to scalars. Its identity component is a torus, and because the
$u_i$ are defined over $\QQ$ it is $\RR$-split (in fact $\QQ$-split)
and therefore it is contained in the maximal $\RR$-split torus in $G_{Q,l}$,
which is $A_Q$. So some power of $l$ is in $A_Q$, and some power of
$\gamma$ is in $N_Q\rtimes A_Q=R_Q$.

If $\gamma^k\in N_Q$ then $\gamma^k=\exp(u'+\eta(v'))$ and comparing
the $V_Q$ parts in $\gamma^ky=y$ yields $v'+v=v$. Therefore $v'=0$ and
$\gamma^k=\exp(u')\in U_Q$. This completes the proof of the ``only
if'' part.

For the converse (the ``if'' part), for $\gamma\in \Gamma\cap Q$ we write
\[
\gamma = (\exp(u+\eta(v)),a,g,h) \in Q  = N_Q \rtimes (A_Q \times
G'_{Q,l} \times G_{Q,h})
\]
and as usual we write $l=(a,g,h)\in L_Q$. Suppose first of all that
$\gamma\in \Lambda$ and that $l^k\in A_Q$ for some $Q \in
\ParGm$. Since $\gamma\in Q$ it preserves the cone $C = C_Q$. By the
Brouwer fixed point theorem, $\gamma$ preserves a ray $\rho'$ in $\Bar
C_Q$.

We claim that there exists a boundary component $F(P)$, for some
$P\in\ParGm$, fixed by $\gamma$ such that $\gamma$ preserves a ray
$\rho=\RR_{\ge 0}u_\rho$ in the relative interior of $C_P$. This is trivial if
$\dim C_Q = 1$. We shall proceed by induction on $\dim C_Q$.

The ray $\rho'$ is preserved by $\gamma$, with eigenvalue $\lambda$
say, and $\rho'$ belongs to a unique real boundary component $C'$ of
$C$, since $\Bar C$ is the disjoint union of its real boundary
components by~\cite[Proposition II.3.1]{AMRT}. Let $H_\lambda$ be the
$\lambda$-eigenspace of $\gamma$ in $U_Q$. Then $H_\lambda\cap \Bar C$
is a boundary component of $C$ and contains $\rho'$, so $H_\lambda\cap
\Bar C = C'$. Thus the normaliser of $C'$ in $\Aut(C)$ is rational
because it is the normaliser of the rational linear subspace
$H_\lambda$. Therefore by \cite[Corollary II.3.22]{AMRT}, $C'$ is a
rational boundary component. But $\gamma$ preserves $C'$, and $\dim
C'<\dim C$. In particular $\gamma\in P$, so we may assume that such a
fixed ray $\rho$ exists already in $C_Q$, generated by an element
$u_\rho\in U_Q$.

We claim that if $\gamma\in \Gamma$ and $\gamma^k \in (R_Q \setminus
N_Q) \cup U_Q$ for some $Q \in \ParGm$ and $k \in \NN$, then $\gamma$
has a fixed point
\[
y_1 = (\lim\limits_{t \to \infty} \bde_Q(u_t + ix_t),v_1,z_1) \in Y_{\Sigma (Q)} \times V_Q \times D_{Q,h}  = Z_{\Sigma (Q)}.
\]
From the group law \eqref{eq:GroupLawQ} we have
$$
\gamma ^k =  \Big(\exp\Big(u+\eta\big(\sum\limits_{j=0}^{k-1} v^{l^j} \big) \Big), a^k, g^k, h^k \Big).
$$
for an appropriate $u \in \Lie(U_Q)$.

The torsion element $h \in G_{Q,h}$ has a fixed point $z_1 \in
D_{1,h}(Q)$. Moreover the point
$o_\rho=\lim\limits_{t\to\infty}\bde_Q(u_t+it u_\rho) \in X_{\Sigma
  (Q)}$ is fixed by $\gamma$, and
\[
 \gamma(o _\rho,v_1,z_1) = (o_\rho, v+ v_1^{l},z_1) =
 l(o_\rho,v_1,z_1)
\]
for any $v_1\in \Lie(V_Q)$. So it is enough to show that there exists
$v_1 \in \Lie(V_Q)$ with
\begin{equation}\label{eq:GammaFixedPoint}
v+v_1^{l} = v_1
\end{equation}
if $l^k = a^k \in A_Q \setminus \{ 1 \}$ or $\gamma^k \in U_Q$.  With
a slight abuse of notation, we identify the split component $A_Q$ of
$Q$ with $(\RR_{>0},.)$ and recall that it acts on $\Lie(V_Q)$ by
scalar multiplication, $(a,v) \mapsto av$.

If $\gamma$ fixes $y_1$ then so does $\gamma^k$, so if $v_1$ satisfies
\eqref{eq:GammaFixedPoint} then it also satisfies
\begin{equation}   \label{eq:GammaKFixedPoint}
  \sum\limits_{j=0}^{k-1} v^{l^j} + a^k v_1 = v_1.
\end{equation}
In the case of $a^k \in A_Q \setminus \{ 1 \}$ we may simply take
$v_1= \frac{1}{1 - a^k} \sum\limits_{j=0}^{k-1} v^{l^j}$: it is
straightforward to verify that this does satisfy \eqref{eq:GammaFixedPoint}.

If $\gamma^k \in U_Q$ then clearly $\gamma^k(o_\rho,v_1,z_1) =
(o_\rho,v_1,z_1)$ for any $v_1 \in \Lie(V_Q)$.  To show that $\gamma$
itself has a fixed point, we note that the adjoint action of $L_Q$ on
$\Lie(V_Q)\imic \CC^n$, given by $l\colon v\mapsto \Ad(l)(x) =
lvl^{-1}$, is $\RR$-linear. Since $\Ad(l)^k = \Id_{\Lie(V_Q)}$, both
$\Ad(l)$ and $\cL = \Id_{\Lie(V_Q)} - \Ad(l)$ are semi-simple.  Thus
$\Ker(\cL) \cap \Im(\cL) = \{ 0 \}$ and $\Lie(V_Q) = \Ker(\cL) \oplus
\Im(\cL)$.  Note that \eqref{eq:GammaFixedPoint} is equivalent to $v =
\cL(v_1)$, so it suffices to prove that $v \in \Im(\cL)$.  Since
$\gamma^k \in U_Q$, {i.e.} its $V_Q$ component vanishes, we have
\begin{equation}    \label{eq:VanishingVQComponent}
\sum\limits_{j=0}^{k-1}  \Ad (l^j)(v) =0.
\end{equation}
We decompose $v = v' + v''$ into $v' \in \Ker(\cL)$ and $v'' \in \Im
(\cL)$: then $0 = \cL (v') = v' -\Ad(l)(v')$ implies $\Ad(l^j)(v') =
v'$ for all $j \geq 0$. Hence
\begin{eqnarray*}
\Im (\cL) \ni \sum\limits_{j=0}^{k-1} \Ad(l^j) (v'')
&=& \sum\limits_{j=0}^{k-1} \Ad (l^j) (v' + v'') - \sum\limits_{j=0}^{k-1} \Ad (l^j) (v') \\
&=& \sum\limits_{j=0}^{k-1} \Ad (l^j) (v) - k v' =  - k v' \in \Ker (\cL),
\end{eqnarray*}
making use of \eqref{eq:VanishingVQComponent}.
Therefore $v' =0$ and $v = v'' \in \Im(\cL)$.

This concludes  the proof of Proposition~\ref{prop:interiorcompanionfixedpoints}.
\end{proof}

\subsection{The first Betti number}\label{subsect:H1}

We can use Theorem~\ref{thm:fundamentalgroup} to give bounds on the
first Betti number of the toroidal compactifications.

\begin{corollary}\label{eq:FirstHomologyGroup}
Suppose that $D= G/K$ is an irreducible  Hermitian symmetric space of non-compact
type with $\dim_\CC (D) > 1$ and $\Gamma$ is a non-uniform
lattice of $G$. Let $Q_1, \ldots, Q_h$ be a complete set of representatives of
the $\Gamma$-conjugacy classes of $\Gamma$-rational maximal
parabolic subgroups of $G$, with solvable radicals $R_{Q_j}$. Then
\begin{equation}\label{eq:RankInequalityHomologyGroups}
\rk_\ZZ H_1 (D/\Gamma,\ZZ) - \sum\limits_{j=1}^h \dim_\RR(R_{Q_j}) \leq \rk_\ZZ H_1(\DGmStor,\ZZ)\leq
\rk_\ZZ H_1 (D/\Gamma,\ZZ).
\end{equation}
If $\Gamma$ is neat then
\[
\rk_\ZZ H_1 (\DGmStor, \ZZ) = \rk_\ZZ H_1 (D/\Gamma , \ZZ).
\]
\end{corollary}

\begin{proof}
For an arbitrary group $\fG$ we denote by $\ab\fG$ its abelianisation
$\ab (\fG) = \fG/[\fG, \fG]$. If $S$ is a complex analytic space then
$H_1 (S, \ZZ) = \ab \pi_1 (S)$.

If $\fH$ is a normal subgroup of $\fG$ then $[\fG/\fH, \fG/\fH] =
[\fG, \fG] \fH/\fH$ and
\[
\ab (\fG/\fH) = (\fG/\fH)/([\fG , \fG] \fH/\fH) \imic \fG/[\fG, \fG]
\fH = \fG/\fH [\fG, \fG].
\]
Therefore by Theorem~\ref{thm:fundamentalgroup}
\[
H_1 (\DGmStor, \ZZ) \imic \ab (\Gamma/\GmU \GmL) \imic
  \Gamma/\GmL\GmU \comGm.
\]
On the other hand, $D$ is a path connected, simply connected locally
compact space with a properly discontinuous action of $\Gamma$ by
homeomorphisms.  Let $\GFix$ be the subgroup of $\Gamma$ generated by
the elements $\gamma \in \Gamma$ with a fixed point on $D$.
By~\cite{Armstrong}, the fundamental group of $D/\Gamma$ is
$\pi_1(D/\Gamma) = \Gamma/\GFix$.  Therefore $H_1 (D/\Gamma , \ZZ)
\imic \ab (\Gamma/\GFix) \imic \Gamma/\GFix \comGm $ and
\[
  H_1 (\DGmStor, \ZZ) \imic
  (\Gamma/\GFix \comGm)/(\GmL\GmU  \comGm/\GFix \comGm) \imic
  H_1 (D/\Gamma, \ZZ)/F
\]
for the abelian group
\[
  F = \GmL\GmU \comGm/\GFix \comGm <
  \Gamma/\GFix \comGm \imic H_1 (D/\Gamma , \ZZ).
\]
In particular,
\[
  \rk_\ZZ H_1 (D/\Gamma, \ZZ) = \rk_\ZZ H_1 (\DGmStor, \ZZ) + \rk_\ZZ F.
\]
To verify~\eqref{eq:RankInequalityHomologyGroups}, it suffices to show
that $F_o = \GFix\GmU \comGm /\GFix \comGm $ is a finite subgroup of
$F$ and that $\rk_{\ZZ} (F) = \rk_{\ZZ} (F/F_o) \leq
\sum_{j=1}^h\dim_\RR (R_{Q_j})$.

We check the rank condition first. For any $Q\in\ParGm$ we define
$\Lambda_Q$ to be the subgroup of $\Gamma$ generated by all
$\gamma \in \Gamma \cap  Q$ such that $\gamma^k\in R_Q$ and
$\gamma^k\in U_Q$ if $\gamma^k\in N_Q$ for some $k \in \NN$. We define
$\Lambda^R_Q$ to be the subgroup of $\Gamma \cap R_Q$ generated by
such $k$-th powers.

Consider the finitely generated abelian group $\fL_Q = \Lambda_Q \GmU
\comGm/\GFix \GmU \comGm$ and the subgroup $\fA_Q = \Lambda^R_Q \GmU
\comGm/\GFix \GmU \comGm$. These both have the same rank, because the
quotient $\fL_Q/\fA_Q$ is abelian and generated by torsion elements so
it is finite.

Now $F/F_o$ is generated by the $\fL_{Q_i}$, and
therefore
\[
\rk_\ZZ (F) = \rk_\ZZ (F/F_o) \leq \sum\limits_{i=1}^h \rk_\ZZ (\fL_{Q_i}) =
\sum\limits_{i=1}^h \rk_\ZZ (\fA_{Q_i}).
\]
However, $\fA_Q$ is a discrete subgroup of $R_Q$, so
$\rk_\ZZ\fA_Q<\rk_\RR R_Q$, and this gives the bound
on $\rk_\ZZ(F)$.

It remains to show that $F_o$ is finite. It is certainly finitely
generated, because it is generated by all
$\GmU_Q\comGm\GFix/\comGm\GFix$; but each $\GmU_Q$ is finitely
generated and $\GmU_{Q^\gamma}=\GmU_Q$ so we need only the
$\GmU_{Q_i}$.

Since $F_o$ is abelian, it is now enough to show that any element of
$F_o$ is of finite order. For $Q\in\ParGm$, consider the group $\Nu_Q
= \Gamma \cap N_Q$: we have
\[
[\Nu_Q, \Nu_Q] \leq \Gamma \cap [N_Q, N_Q] = \Gamma \cap U_Q = \GmU_Q.
\]
It suffices to prove that $\Span_\RR [\Nu_Q, \Nu_Q] = U_Q$, because
then $[\Nu_Q, \Nu_Q]$ is of finite index in $\GmU_Q$ so
$\GmU_Q/[\Nu_Q,\Nu_Q]$ is finite, of exponent $k_Q$ say. Then
$(\GmU_Q)^{k_Q}\leq [\Nu_Q,\Nu_Q]\leq \comGm$ and any element of
$\GmU_Q \GFix \comGm/\GFix \comGm $ is of order dividing $k_Q$.

To prove $\Span_\RR [\Nu_Q, \Nu_Q] = U_Q$, note that the group $N_Q$
is $2$-step nilpotent, so $[\Nu_Q,\Nu_Q]\subset U_Q$ and hence
$\Span_\RR [\Nu_Q, \Nu_Q] \subseteq U_Q$. For the other inclusion, let
$\beta_1, \ldots , \beta_{m + 2n} \in \Lie(N_Q)$ be such that $b_j =
\exp (\beta_j) \in N_Q$ generate the lattice $\Nu_Q$.  Then $\Lie(N_Q)
= \Span_\RR (\beta_1, \ldots , \beta_{m + 2n})$ and
\[
 \Lie(U_Q) = [\Lie(N_Q), \Lie(N_Q)] = \Span_\RR (\{[\beta_i,
   \beta_j]\}).
\]
Here $[\beta_i,\beta_j]$ is the Lie bracket, but $U_Q$ is isomorphic
to $\Lie(U_Q)$ via $\exp$, so that $\exp[\beta_i, \beta_j] = [\exp
  (\beta_i), \exp (\beta_j)] = [b_i, b_j]$. Thus
\[
U_Q = \Span_\RR (\{[b_i, b_j]\})\subseteq \Span_\RR[\Nu_Q,\Nu_Q],
\]
as required.

For neat $\Gamma$, it is shown in \cite{Sankaran} that
$(\Gamma/\GmU)^\Fix_o$ is trivial.  Combining this with
$(\Gamma/\GmU)^\Fix_o = \GmL \GmU/\GmU$ from
Lemma~\ref{prop:interiorcompanionfixedpoints}, one concludes that
$\GmL \subseteq \GmU$.  Therefore $F = F_o$ and $\rk_\ZZ (F) =0$.
\end{proof}

In the case of a classical irreducible Hermitian symmetric space $D =
G/K$ of non-compact type, Wolf's book \cite{Wolf} provides matrix
realisations of the maximal parabolic subgroups $Q$ of $G$ and their
solvable radicals $R_Q$.  These allow an explicit calculation of $\dim
_{\RR} (R_Q)$, depending on the parameters of $G$ and the parameters
of the associated complex analytic boundary component $F(Q)$ of $Q$.

\end{document}